\documentclass[11pt,reqno]{article}

\usepackage{amsfonts}
\usepackage{amscd}
\usepackage{color}
\usepackage{todonotes}
\usepackage{amsmath}
\usepackage{float}
\usepackage{amsthm}
\usepackage{amssymb}
\usepackage{mathrsfs}
\usepackage{amstext}
\usepackage{graphicx}
\usepackage{tikz}
\usepackage{verbatim}
\usepackage{mathtools}  
\mathtoolsset{showonlyrefs}

\usepackage{todonotes}

\numberwithin{equation}{section}

\theoremstyle{definition}
	
	\newtheorem{theorem}{Theorem}[section]
	
	\newtheorem*{remark}{Remark}

\theoremstyle{plain}
\newtheorem{satz}{Theorem}[section]
\newtheorem{defi}[satz]{Definition}
\newtheorem{cor}[satz]{Corollary}
\newtheorem{lem}[satz]{Lemma}
\newtheorem{prop}[satz]{Proposition}
\newtheorem{rem}[satz]{Remark}

\newcommand{\supp}{\operatorname{supp}}

\newcommand{\re}{\ensuremath{\mathbb{R}}}
\newcommand{\R}{\ensuremath{{\re}^d}}

\newcommand{\N}{\ensuremath{\mathbb{N}}}
\newcommand{\n}{\ensuremath{{\N}_0}}

\newcommand{\zz}{\ensuremath{\mathbb{Z}}}
\newcommand{\Z}{{\ensuremath{\zz}^d}}

\newcommand{\C}{\ensuremath{\mathbb{C}}}

\newcommand{\D}{\ensuremath{\mathbb{D}}}

\newcommand\dint{{\rm d}}

\newcommand{\cF}{\ensuremath{\mathcal F}}

\newcommand{\eps}{\varepsilon}

\newcommand{\bproof}{\begin{proof}}
\newcommand{\eproof}{\end{proof}}

\newlength{\fixboxwidth}
\newcommand{\be}{\begin{equation}}
\newcommand{\ee}{\end{equation}}
\newcommand{\beq}{\begin{eqnarray}}
\newcommand{\beqq}{\begin{eqnarray*}}
\newcommand{\eeq}{\end{eqnarray}}
\newcommand{\eeqq}{\end{eqnarray*}}

\begin{document}
\title{$L_p$-Sampling recovery for non-compact subclasses of $L_\infty$}

\author{Glenn Byrenheid $\!\!{}^{\dag}$, 
Serhii Stasyuk $\!\!{}^{\ddag,a}$, 
Tino Ullrich $\!\!{}^{\ddag}$\footnote{Corresponding author, Email:
tino.ullrich@math.tu-chemnitz.de}\\\\
$\!\!{}^{\dag}\!\!$ Friedrich-Schiller-University Jena, Institute of Mathematics\\\\
$\!\!{}^{\ddag}\!\!$ Chemnitz University of Technology, Faculty of Mathematics\\\\
$\!\!{}^{a}\!\!$ Institute of Mathematics of NAS of Ukraine}

\date{\today}

\maketitle

\begin{abstract} In this paper we study the sampling recovery problem for certain relevant multivariate function classes which are not compactly embedded into $L_\infty$. Recent tools relating the sampling numbers to the Kolmogorov numbers in the uniform norm are therefore not applicable. In a sense, we continue the research on the small smoothness problem by considering ``very'' small smoothness in the context of Besov and Triebel-Lizorkin spaces with dominating mixed regularity. There is not much known on the recovery of such functions except of an old result by Oswald in the univariate situation. As a first step we prove the uniform boundedness of the $\ell_p$-norm of the Faber-Schauder coefficients in a fixed level. Using this we are able to control the error made by a (Smolyak) truncated Faber-Schauder series in $L_q$ with $q<\infty$. It turns out that the main rate of convergence is sharp. As a consequence we obtain results also for $S^1_{1,\infty}F([0,1]^d)$, a space which is ``close'' to the space $S^1_1W([0,1]^d)$ which is important in numerical analysis, especially numerical integration, but has rather bad Fourier analytical properties.  
\end{abstract}



\section{Introduction} In this paper we continue studying the approximation power of Smolyak sparse grid sampling recovery for multivariate classes  with mixed smoothness $S^r_{p,q}B([0,1]^d)$ and $S^r_{p,q}F([0,1]^d)$ based on the Faber-Schauder representation \cite{By18,HiMaOeUl16} of $f$
$$
    f = \sum\limits_{j\in \N_{-1}^d}\sum\limits_{k \in \mathbb{D}_j}d^2_{j,k}(f)v_{j,k}\,.
$$
Here $v_{j,k}$ represents the tensorized dilated and translated hat function ($L_\infty$-nomalized) and 
$$
    d^2_{j,k}(f) = \prod\limits_{i=1}^d \left\{\begin{array}{rcl}
            \Delta^{2,i}_{2^{-j_i-1}}&:&j_i\geq 0\\
            {\rm Id}&:&j_i = -1
    \end{array}\right\}f(k_12^{-(j_1)_+},...,k_d2^{-(j_d)_+})
$$
represent the Faber-Schauder coefficients of $f$. The idea to approximate a function $f$ by a truncated Faber-Schauder series
\begin{equation}\label{f13_b}
    I_n f :=\sum\limits_{|j|_1\leq n} \sum_{k}d^2_{j,k}(f)v_{j,k}
\end{equation}
has been proposed by Triebel \cite{Tr10} and, independently, Dinh D\~ung \cite{Du11}. It provides a powerful tool for studying various situations of errors measured in $L_q$. The latter reference is even more general and uses so-called B-spline representations of functions, where the Faber system is a special case. A systematic discretization of multivariate functions with mixed smoothness in terms of Faber coefficients has been given \cite{HiMaOeUl16} and \cite{By18}, see also \cite{DuTeUl18} for further history. Very recently, the authors in \cite{KeSchVy22} used the 
Faber system to analyze the path regularity Brownian motion.  
Surprisingly, the proposed method turned out to be sharp in several regimes. In this paper we study an endpoint situation, where $r=1/p$ and $q\leq 1$. This still allows for an embedding into the continuous functions (and makes therefore function evaluations possible). However, the embedding is not compact.  

Recent observations regarding the problem of optimal sampling recovery of function classes in $L_2$ bring classes with mixed smoothness to the focus again. Since several newly developed techniques only work for Hilbert-Schmidt operators \cite{KrUl19}, \cite{NaSchUl21}, \cite{BSU22} or, more generally, in situations where certain asymptotic characteristics (approximation numbers) are square summable \cite{KrUl20}, \cite{DKU22} we need new techniques in situations where this is not the case. In \cite{TeUl22, TeUl21} the range of small smoothness has been considered where one is far away from square summability of corresponding widths. Still we have the compact embedding into $L_\infty$ for those examples. This embedding seems to be of crucial importance since the sampling numbers in $L_2$ can be related to the Kolmogorov numbers in $L_\infty$. It has been shown in \cite{Te21,BSU22}
\begin{equation}\label{f1}
    \varrho_{bm}({\mathbf F},L_2) \leq \frac{C}{(b-1)^{3/2}}d_m({\mathbf F},L_\infty)\,.    
\end{equation}
with some absolute constant $C>0$ and an oversampling constant $b>1$. In other words, a compact embedding into $L_\infty$ allows us to use the decaying Kolmogorov numbers for controlling the sampling numbers. As discussed in \cite{TeUl22} sparse grid techniques \cite{BuGr04} perform asymptotically worse by a $\log$-factor for classes with mixed smoothness compactly embedded into $L_\infty$.

In this paper we continue research in this direction. Note, that there are several relevant (multivariate) function classes ${\mathbf F}$ which are continuously embedded in $L_\infty$ but not compactly. Thus, \eqref{f1} is not applicable. Almost nothing is known how to reconstruct functions from samples, which only satisfy a Besov regularity with smoothness $r=1/p$ or Sobolev type regularity with $r=1$ and $p=1$. As a first step we prove the following relation for the Faber-Schauder coefficients
$$
    \sup\limits_{j \in \N_{-1}^d} \Big(\sum\limits_{k\in \Z^d} |d^2_{j,k}(f)|^p\Big)^{1/p} \lesssim \|f|S^{1/p}_{p,1}B(\R)\|\,.
$$
This relation directly implies that the truncated Faber-Schauder representation \eqref{f13_b} still works well when we consider errors in $L_q$ with $q<\infty$. We make progress towards the solution of an open problem mentioned in \cite{TDiss07}. 
The univariate class $B^{1/p}_{p,1}$ and its approximation by equidistant samples in $[0,1]$ has been considered decades ago by Oswald \cite{Os81}.
Our main results read as follows. In case $1\leq p<\infty$ there is a constant $C$ depending on $p$ and $d$ such that 
$$\|f-I_nf\|_p\leq C 2^{-n/p}n^{d-1}\|f|S^{1/p}_{p,1}B([0,1]^d)\|$$
holds for all $n\in \mathbb{N}$. In the case $1\leq p<q<\infty$ we are able to prove the following bound. 
$$\|f-I_nf\|_q\leq C 2^{-n/q}n^{(d-1)/q}\|f|S^{1/p}_{p,1}B([0,1]^d)\|\,.$$
Due to embeddings the latter bound also holds true for $S^{1}_{1,\infty}F([0,1]^d)$\,. All these bounds can be rephrased in terms of linear sampling numbers. In case $1\leq p<q <\infty$ we obtain, for instance, 
$$
        m^{-1/q} \lesssim \varrho_m^{\rm lin}(S^{1/p}_{p,1}B([0,1]^d),L_q([0,1]^d)) \lesssim 
        (m^{-1}\log^{2(d-1)} m)^{1/q}\,.
$$

{\bf Notation.} As usual $\N$ denotes the natural numbers, $\N_0=\N\cup\{0\}$,
$\N_{-1}=\N_0\cup\{-1\}$, $\zz$ denotes the integers, 
$\re$ the real numbers, and $\C$ the complex numbers. The letter $d$ is always
reserved for the underlying dimension in $\re^d, \zz^d$ etc. For $a\in \re$ we denote $a_+ := \max\{a,0\}$. 
For $0<p\leq \infty$ and $x\in \R$ we denote $|x|_p = (\sum_{i=1}^d
|x_i|^p)^{1/p}$ with the
usual modification in the case $p=\infty$. We further denote 
$x_+ := ((x_1)_+,\ldots,(x_d)_+)$ and $|x|_+ := |x_+|_1$. By $(x_1,\ldots,x_d)>0$
we mean that each coordinate is positive. If $X$ and $Y$ are two (quasi-)normed spaces, the (quasi-)norm
of an element $x$ in $X$ will be denoted by $\|x|X\|$. 
The symbol $X \hookrightarrow Y$ indicates that the
identity operator is continuous. For two sequences $a_n$ and $b_n$ we will write
$a_n \lesssim b_n$ if there exists a constant $c>0$ such that $a_n \leq c\,b_n$
for all $n$. We will write $a_n \asymp b_n$ if $a_n \lesssim b_n$ and $b_n
\lesssim a_n$.


\section{The tensor Faber bases}
\label{HaarFaber}
As a main tool we will use decompositions of functions in terms of a Faber series expansion. 

\begin{figure}[H]\label{fig:3}
\centering
\begin{tikzpicture}[scale=3]

\draw[->] (-0.1,0.0) -- (1.1,0.0);
\draw[->] (0.0,-0.1) -- (0.0,1.1); 

\draw (1.0,0.03) -- (1.0,-0.03) node [below] {$1$};
\draw (0.03,1.0) -- (-0.03,1.00) node [left] {$1$};
\draw[->] (0.8,0.9) -- (0.7,0.7);

\node at (0.8,1) {$v_{0,0}$};

\node at (1.1,0.5) {$j=0$};

\draw (0,0) -- (0.5,1);
\draw (0.5,1) -- (1,0);

\draw[->] (1.4,0.0) -- (2.6,0.0); 
\draw[->] (1.5,-0.1) -- (1.5,1.1); 

\draw (2.5,0.03) -- (2.5,-0.03) node [below] {$1$};
\draw (1.53,1.0) -- (1.47,1.00) node [left] {$1$};
\draw[->] (2.5,0.9) -- (2.4,0.7) ;
\draw[->] (2,0.9) -- (1.9,0.7) ;

\node at (2.5,1) {$v_{1,1}$};
\node at (2,1) {$v_{1,0}$};

\node at (2.6,0.5) {$j=1$};

\draw (1.5,0,0) -- (1.75,1.0);
\draw (1.75,1) -- (2,0.0);
\draw (2,0) -- (2.25,1.0);
\draw (2.25,1) -- (2.5,0.0);

\draw[->] (2.9,0.0) -- (4.1,0.0); 
\draw[->] (3.0,-0.1) -- (3.0,1.1);

\draw (4.0,0.03) -- (4.0,-0.03) node [below] {$1$};
\draw (3.03,1.0) -- (2.97,1.00) node [left] {$1$};

\fill[gray!60] plot[domain=3:3.5] (\x,-6+2*\x)%
                      -- plot[domain=4:3.5] (\x,0);%

\node at (4.3,0.5) {$j\in\{0,1\}$};

\draw (3,0) -- (3.5,1);
\draw (3.5,1) -- (4,0);

\draw (3.0,0,0) -- (3.25,1.0);
\draw (3.25,1) -- (3.5,0.0);
\draw (3.5,0) -- (3.75,1.0);
\draw (3.75,1) -- (4,0.0);

\end{tikzpicture}
  \caption{Univariate hierarchical Faber basis on $[0,1]$ for levels $j\in\{0,1\}$ and their union.} \label{fig_Faber1}
\end{figure}
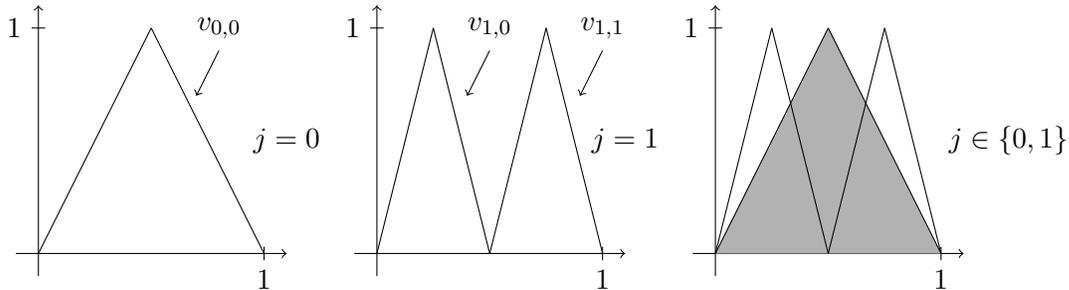

\subsection{The univariate Faber basis}

Let us briefly recall the basic facts about the Faber basis taken from
\cite[3.2.1, 3.2.2]{Tr10}. For $j\in \N_0$ and $k \in \D_j := \{0,1,...,2^j-1\}$ we denote by $I_{j,k}$ the dyadic interval
$$
    I_{j,k} = [2^{-j}k,2^{-j}(k+1)]\,.
$$ 
\begin{defi}\label{def:haar_faber}[The univariate Faber system] Let 
$$
      h(t) = \left\{\begin{array}{rcl}
                    1&:&t\in [0,1/2]\,,\\
                    -1&:&t\in [1/2,1]\,,\\
                    0&:& \mbox{otherwise}\,.
                \end{array}\right.
$$
the Haar function and $v(x)$ be the integrated Haar
function, i.e., 
\be\label{vinth}
  v(x):=2\int_0^x h(t)\,\dint t\quad,\quad x\in \re\,,
\ee
and for $j\in \n, k\in \D_j$  
\be\label{f19}
      v_{j,k}(\cdot) = v(2^j\cdot - k)\,.
\ee
For notational reasons we let $v_{-1,0}:=x$ and $v_{-1,1} := 1-x$ for $j=-1$
and obtain the Faber system
$$
    F:=\{v_{j,k}:j\in \N_{-1}, k\in \D_j\}\,,
$$
where $\D_{-1}:=\D_{1}= \{0,1\}$.
\end{defi}

Faber \cite{Fa09} observed that every continuous
(non-periodic) function $f$ on $[0,1]$ can be represented as
\begin{equation}\label{f51}
    f(x) = f(0)\cdot (1-x)+f(1)\cdot x -
\frac{1}{2}\sum\limits_{j=0}^{\infty}\sum\limits_{k=0}^{2^j-1}
\Delta^2_{2^{-j-1}}(f,2^{-j}k)v_{j,k}(x)
\end{equation}
with uniform convergence, see e.g. \cite[Theorem 2.1, Step 4]{Tr10}. The analysis of Triebel-Lizorkin and Besov spaces on $\re$ as defined in section \ref{perfs} requires a version of the Faber-Schauder representation acting on $\re$. For this purpose we extend the number of translations to the whole integers and obtain 
\be  f(x) = \sum_{k\in \zz} f(k) v_{-1,k}(x)-
\frac{1}{2}\sum\limits_{j=0}^{\infty}\sum\limits_{k\in \zz}
\Delta^2_{2^{-j-1}}(f,2^{-j}k)v_{j,k}(x),\label{repronr}\ee
where $v_{-1,k}(\cdot):=v_{0,0}((\cdot+1+k)/2)$.

\subsection{The tensor Faber basis}
Let now $f(x_1,\ldots,x_d)$ be a $d$-variate function $f\in C(\R)$. By fixing
all variables except $x_i$
we obtain by $g(\cdot) = f(x_1,\ldots, x_{i-1},\cdot,x_{i+1},\ldots,x_d)$ a univariate
continuous function. By applying \eqref{repronr} in every such component we
obtain the representation
\begin{equation}\label{repr_C(R^d)}
  f(x) = \sum\limits_{j\in \N_{-1}^d} \sum\limits_{k\in \Z} d^2_{j,k}(f)
  v_{j,k}(x)
\end{equation}
 in $C(K)$, $K\subset \R$ compact,
where 
$$
  v_{j,k}(x_1,\ldots,x_d):=v_{j_1,k_1}(x_1)\cdots
v_{j_d,k_d}(x_d)\quad,\quad j\in \N_{-1}^d, k\in \zz^d\,,
$$
and 
\begin{equation}\label{f100}
      d^2_{j,k}(f) =
(-2)^{-|e(j)|}\Delta^{2,e(j)}_{2^{-(j+1)}}(f,x_{j,k})\quad,\quad j\in \N_{-1}^d,
k\in \Z\,.
\end{equation}
Here we put $e(j) = \{i:j_i \neq -1\}$ and $x_{j,k} =
(2^{-(j_1)_+}k_1,\ldots,2^{-(j_d)_+}k_d)$\,.\\
In Section \ref{application} we apply the Faber-Schauder series expansion for functions on the  $d$-variate unit cube $[0,1]^d.$ For this purpose we simply truncate the series expansion to all translations whose support has a non-empty intersection with $[0,1]^d$. That is  
$$ 
 \D_j:=\D_{j_1} \times \cdots \times \D_{j_d}\quad,\quad j = (j_1,\dots,j_d)\in \N_{-1}^d\,.
$$ 
With similar arguments as above we obtain tor $f\in C([0,1]^d)$ the representation 
\begin{equation}\label{repr_C([0,1]^d)}
  f(x) = \sum\limits_{j\in \N_{-1}^d} \sum\limits_{k\in \D_j} d^2_{j,k}(f)
  v_{j,k}(x).
\end{equation}

\section{Faber-Schauder coefficients and bandlimited functions}

In the sequel we denote with $\mathbb{I}^d$ one of the tensor domains $[0,1]^d$ or $\R$. The space $L_p(\mathbb{I}^d)$, $0<p\leq \infty$, contains all measurable functions $f:\mathbb{I}^d\rightarrow \C$ satisfying
$$
    \|f\|_p =
    \Big(\,\int_{\mathbb{I}^d}|f(x)|^p\,dx\Big)^{1/p} < \infty
$$
with the usual modification in case $p=\infty$. The space $C(\mathbb{I}^d)$ is often
used as a replacement for $L_{\infty}(\mathbb{I}^d)$. It denotes the collection of all
continuous and bounded $d$-variate functions equipped with the uniform norm.
The computation of the Fourier transform (and its inverse) of an $L_1$-integrable $d$-variate function is performed by the integrals ($\xi \in \R$)
$$
   \mathcal{F}f(\xi)= \frac{1}{(2\pi)^{d/2}}\int_{\R}
  f(x)e^{-i \xi \cdot x}\,dx\quad,\quad \mathcal{F}^{-1}f(\xi)= \frac{1}{(2\pi)^{d/2}}\int_{\R}
  f(x)e^{i \xi \cdot x}\,dx\,.
$$
To begin with, we recall the concept of a dyadic decomposition of unity. The
space $C^{\infty}_0(\re^d)$ consists of all infinitely many times differentiable compactly supported functions. 
    \begin{defi}\label{cunity} Let $\Phi(\re)$ be the collection of all
      systems
       $\varphi = \{\varphi_n(x)\}_{n=0}^{\infty} \subset C^{\infty}_0(\re)$
       satisfying
       \begin{description}
       \item(i) ${\supp}\,\varphi_0 \subset \{x:|x| \leq 2\}$\, ,
       \item(ii) ${\supp}\,\varphi_n \subset \{x:2^{n-1} \leq |x|
       \leq 2^{n+1}\}\quad,\quad n= 1,2,\ldots ,$
       \item(iii) For all $\ell \in \N_0$ it holds
       $\sup\limits_{x,n}
       2^{n\ell}\, |D^{\ell}\varphi_n(x)| \leq c_{\ell} <\infty$\, ,
       \item(iv) $\sum\limits_{n=0}^{\infty} \varphi_n(x) = 1$ for all
       $x\in \re$.
       \end{description}
    \end{defi}
    \noindent Now we fix a system $\varphi=\{\varphi_n\}_{n\in \zz} \in
\Phi(\re)$, where we 
    put $\varphi_n \equiv 0$ if $n<0$. For $\ell = (\ell_1,\ldots,\ell_d) \in \zz^d$ let the
building blocks $f_{\ell}$ be given by
     \begin{equation}\label{f2}
	f_\ell(x) = \cF^{-1}[\varphi_{\ell_1}(\xi_1)\cdots\varphi_{\ell_d}(\xi_d)\cF f(\xi)](x)
\quad,\quad x\in \R\,.
    \end{equation}
Because of the Paley-Wiener theorem the functions $f_\ell$ are entire analytic functions and therefore continuous. The goal of this section is to derive bounds for fixed ``levels'' $j$ of the Faber-Schauder expansion \eqref{repr_C(R^d)} of such a bandlimited function $f_{\ell}$. To be more precise, we aim at bounds for 
$\|\sum_{k\in \zz^d}d^2_{j,k}(f_{\ell})v_{j,k}(\cdot)\|_p\,.$
Clearly, due to the compact support of $v_{j,k}$ we may replace $v_{j,k}$ by 
the characteristic function $\chi_{j,k}$ of the parallelepiped $[2^{-j_1}k_1,2^{-j_1}(k_1+1)]\times \cdots \times [2^{-j_d}k_d,2^{-j_d}(k_d+1)]$\,. Note, that for any continuous function $f \in C(\R)$
$$
    \Big\|\sum\limits_{k\in \zz^d}d^2_{j,k}(f)v_{j,k}(\cdot)\Big\|_p\asymp 
    \Big\|\sum\limits_{k\in \zz^d}d^2_{j,k}(f)\chi_{j,k}(\cdot)\Big\|_p\asymp
    \Big(2^{-|j|_1}\sum\limits_{k\in \zz^d}|d^2_{j,k}(f)|^p\Big)^{1/p}\,.
$$
To do so we need some tools from harmonic analysis. We state a mixed version of the Peetre maximal inequality, proveed in \cite[1.6.4]{SchTr87}.

\begin{lem}\label{Pmax}[Peetre maximal inequality]\label{def:peetremax}
	Let $0<p<\infty$ and $a>1/p$. Let further $f \in L_1(\R)$ such that $\supp \cF f \subset [-b_1,b_1]\times \cdots \times [-b_d,b_d]$. Then there is a constant $c>0$, only depending on $a$ and $p$ but not on $f$ and $b$, such that 
	$$
	\Big\|\sup_{ y\in \R}\frac{|f( x+ y)|}{(1+b_1|y_1|)^{a}\ldots (1+b_d|y_d|)^{a}}\Big\|_p \leq c\|f\|_p\,.
	$$
	
\end{lem}

The following univariate pointwise estimate connecting differences of bandlimited functions to Peetre maximal operator is taken from \cite[Lem.\ 3.3.1]{TDiff06}. \begin{lem}\label{lem:diffversuspeetre}
	Let $a,b>0$ and $f\in L_1(\re)$ with $ \supp \mathcal{F}f\subset[-b,b]$. 
	Then there exists a constant $C>0$ such that
	\be 
	|\Delta_h^mf(x)|\leq C\min\{1,|bh|^m\}\max\{1,|bh|^a\}\sup\limits_{y\in \re}\frac{|f(x+y)|}{(1+|2^{j+\ell}y|)^a}\,.
	\ee

\end{lem}

The bound may be slightly improved when replacing pointwise estimates by estimates involving $L_p$ norms. We have the following. 

\begin{lem}\label{sharp} Let $j,\ell\in \N_0$ and $f \in C(\re)$. Then 
$$
    \|\sup\limits_{|h|\lesssim 2^{-j}}|f(\cdot +h)|\|_p \lesssim 2^{\ell/p}\|\sup_{|h|\lesssim 2^{-(j+\ell)}}|f(x+h)|\|_p
    \lesssim 2^{\ell/p}\Big\|\sup\limits_{y\in \re}\frac{f(x+y)}{(1+|2^{j+\ell}y|)^a}\Big\|_p\,.
$$
\end{lem}

\begin{proof} We start with a pointwise estimate 
\begin{equation}
  \begin{split}
    \sup_{|h|\leq 2^{-j}}|f(x+h)| &\leq \sup_{k}\sup_{|h|
    \leq 2^{-(j+\ell)}}|f(x+k2^{-(j+\ell)}+h)|\\ &\leq \Big(\sum_{k} \sup_{|h|\lesssim 2^{-(j+\ell)}}|f(x+k2^{-(j+\ell)}+h) |^p\Big)^{1/p}\,.
  \end{split}    
\end{equation}
Taking $L_p$-norms on both sides  gives
\begin{equation}
   \begin{split}
    \|\sup_{|h|\leq 2^{-j}}|f(x+h)|\|_p &\leq
    \Big(\sum\limits_k 
    \int_{\re}\sup_{|h|\leq 2^{-(j+\ell)}}|f_{j+\ell}(x+k2^{-(j+\ell)}+h) |^p\,dx\Big)^{1/p}\\
    &\lesssim 2^{\ell/p}\|\sup_{|h|\lesssim 2^{-(j+\ell)}}|f(x+h)|\|_p\,.
   \end{split}
\end{equation}
\end{proof}

In the next lemma we combine both univariate bounds and derive a multivariate estimate via iteration with respect to coordinate directions. Let $f\in L_1(\R)$ and $f_{j+\ell}$ denotes the bandlimited function from \eqref{f2}.  
\begin{lem}\label{lem:estfsblock} Let $j\in \N_{-1}^d$ and $\ell \in \zz^d$. Then we have
$$
    \Big\|\sum\limits_{k\in \zz^d}d^2_{j,k}(f_{j+\ell})\chi_{j,k}(\cdot)\Big\|_p
    \lesssim\|f_{j+\ell}\|_p \prod\limits_{i =1}^d \min\{2^{2\ell_i},1\}\max\{2^{\ell_i/p},1\}\,.
$$

\end{lem}

\begin{proof} {\em Step 1.} To provide a technically transparent proof of this Lemma we start with the univariate case $(d=1)$. In the second part of this proof we deal with the multivariate case which requires more involved notation. We put for $x\in \re$
$$
   F_{j,l}(x):=\sum\limits_{k\in \mathbb{Z}} d^2_{j,k}(f_{j+l})\chi_{j,k}(x)\,.
$$
Let $x\in [2^{-j}k,2^{-j}(k+1)]$. For this $x$ we have
   
$$
|d^2_{j,k}(f_{j+l})\chi_{j,k}(x)|
\lesssim |f_{j+l}(k\cdot 2^{-j}+2\cdot 2^{-(j+1)})|
+2|f_{j+l}(k\cdot 2^{-j}+2^{-(j+1)})|
$$
$$
+|f_{j+l}(k\cdot 2^{-j})|
\lesssim \sup\limits_{|h|\lesssim 2^{-j}} |f_{j+l}(x+h)|\,.
$$
Since $\chi_{j,k}(x)$ do not have overlap we receive 
$$
|F_{j,l}(x)|
=
\Big|\sum\limits_{k\in \mathbb{Z}} d^2_{j,k}(f_{j+l})\chi_{j,k}(x)\Big|
\lesssim \sup\limits_{|h|\lesssim 2^{-j}} |f_{j+l}(x+h)|\,.
$$
By Lemma \ref{sharp} we find 
\begin{equation}\label{f11}
    \|F_{j,l}\|_p \lesssim \max\{2^{\ell/p},1\}\Big\|\sup\limits_{y\in \re}\frac{f_{j+l}(x+y)}{(1+|2^{j+\ell}y|)^a}\Big\|_p
\end{equation}
for some $a>0$ which is at our disposal. 

In case $l \leq 0$ we may continue arguing pointwise. First of all, we have
$$
    |F_{j,l}(x)| \leq \sup\limits_{|y|\lesssim 2^{-j}}|\Delta^2_{2^{-j-1}}f_{j+l}(x+y)|\,.
$$
Using Lemma \ref{lem:diffversuspeetre} and the fact that $l\leq 0$ we obtain 
\begin{equation}\label{f13}
     |F_{j,l}(x)| \leq 2^{2l}\sup\limits_{y\in \re}\frac{|f(x+y)|}{(1+|2^{j+\ell}y|)^a}\,.
\end{equation}
Combining \eqref{f11} and \eqref{f13} gives
$$
    \|F_{j,l}(x)\|_p \lesssim \Big\|\sup\limits_{y\in \re}\frac{|f(x+y)|}{(1+|2^{j+\ell}y|)^a}\Big\|_p\max\{1,2^{l/p}\}\min\{1,2^{2 l}\}\,.
$$
Choosing $a>1/p$ and applying the Peetre maximal inequality in Lemma \ref{Pmax} gives the result for $d=1$\,.

{\em Step 2.} We deal with the multivariate case and start with a pointwise estimate of
$$
    F_{j,\ell}(x):=\sum\limits_{k\in \Z} d^2_{j,k}(f_{j+\ell})\chi_{j,k}(x)\,,
$$
where we apply the above procedure in every direction. In order not to drown in notation we introduce the following direction-wise maximal operator
$$
    {\mathfrak M}^i_jf(x) := \sup\limits_{|h|\lesssim 2^{-j}}|f(x+e_ih)| \, ,
$$
where $e_i=(\underbrace{0,\dots,0}_{ i-1 }
,1,
\underbrace{0,\dots,0}_{ d-i })$.
Clearly, we have for $x\in \R$
\begin{equation}
  \begin{split}
   |F_{j,\ell}(x)| \lesssim \Big(\prod\limits_{i:\ell_i>0} {\mathfrak M}_{j_i}^i\Big) \circ \Big(\prod\limits_{i:\ell_i\leq 0} {\mathfrak M}_{j_i}^i\Big) \circ 
    \Big(\prod\limits_{i:\ell_i\leq 0}\Delta^{2,i}_{2^{-j_i-1}}\Big)f(x)\,.
  \end{split}
\end{equation}
Here we use that we have in every direction 
$$
 |\Delta^{2,i}_{2^{-j-1}} f_{j+\ell}| \leq {\mathfrak M}^i_{j_i} f_{j+\ell}(x)
$$
including the case $j=-1$ where the difference is replaced by the function value at the respective point. This case is included in $\Big(\prod_{i:\ell_i>0} {\mathfrak M}_{j}^i\Big)$ since for $\ell<0$ there is nothing to prove in this case. We use the triangle inequality in order to estimate the difference by point evaluations. Taking the $\sup$ leads to the direction-wise maximal operator. In case $\ell\leq 0$ we keep the direction-wise difference. Nevertheless, in order to get rid of the characteristic function we have to additionally apply a direction-wise $\sup$ also in this case. Since it holds that 
$$
    {\mathfrak M}_{j_i}^i f(x) \leq{\mathfrak M}_{j_i}^i g(x) 
$$
with $|f| \leq |g|$ we first estimate 
$$
    \Big|\Big(\prod\limits_{i:\ell_i\leq 0}\Delta^{2,i}_{2^{-j-1}}\Big)f_{j+\ell}(x)\Big|
$$
pointwise from above using Lemma \ref{lem:diffversuspeetre} iteratively. This gives 
$$
    \Big(\prod\limits_{i:\ell_i\leq 0}\Delta^{2,i}_{2^{-j-1}}\Big)f_{j+\ell}(x)
    \lesssim \sup\limits_{\substack{y_i \in \re \\ i:\ell_i \leq 0}}
    \frac{|f_{j+\ell}(x+\sum\limits_{i:\ell_i\leq 0}y_ie_i)|}{\prod\limits_{i:\ell_i\leq  0}(1+2^{\ell_i+j_i}|y_i|)^a}    \prod\limits_{i:\ell_i\leq 0} 2^{2\ell_i}\,.
$$
for any $a>0$\,. The maximal operators $(\prod_{i:\ell_i\leq 0} {\mathfrak M}_{j_i}^i)$ go on the same coordinates. Since $\ell_i\leq 0$ we clearly have also 
$$
        \Big(\prod\limits_{i:\ell_i\leq 0} {\mathfrak M}_{j_i}^i\Big)\Big(\prod\limits_{i:\ell_i\leq 0}\Delta^{2,i}_{2^{-j-1}}\Big)f_{j+\ell}(x)
    \lesssim \sup\limits_{\substack{y_i \in \re \\ i:\ell_i \leq 0}}
    \frac{|f(x+\sum\limits_{i:\ell_i\leq 0}y_ie_i)|}{\prod\limits_{i:\ell_i\leq 0}(1+2^{\ell_i+j_i}|y_i|)^a}
    \prod\limits_{i:\ell_i\leq 0} 2^{2\ell_i}
    \,.
$$
It remains to apply $(\prod_{i:\ell_i>0} {\mathfrak M}_{j_i}^i)$ and take the $L_p$-norm. Here we use Lemma \ref{sharp} iteratively.
\begin{equation}
    \begin{split}
        &\int_{\R} |F_{j,\ell}(x)|^p\,dx \\
        &\lesssim 
            \idotsint\limits_{\{(x_i)_{i:\ell_i\leq 0}\}}~\idotsint\limits_{\{(x_i)_{i:\ell_i> 0}\}}\Big|\Big(\prod\limits_{i:\ell_i>0} {\mathfrak M}_{j_i}^i\Big) \circ \Big(\prod\limits_{i:\ell_i\leq 0} {\mathfrak M}_{j_i}^i\Big) \circ 
    \Big(\prod\limits_{i:\ell_i\leq 0}\Delta^{2,i}_{2^{-j-1}}\Big)f_{j+\ell}(x)\Big|^p\prod\limits_{i:\ell_i>0}dx_i\prod\limits_{i:\ell_i\leq 0}dx_i\\
    &\lesssim \prod\limits_{i =1}^d \min\{2^{2p\ell_i},1\}\max\{2^{\ell_i},1\}\times\\
    &~~~~\times\idotsint\limits_{\{(x_i)_{i:\ell_i\leq 0}\}}~\idotsint\limits_{\{(x_i)_{i:\ell_i> 0}\}}\Big|\Big(\prod\limits_{i:\ell_i>0} {\mathfrak M}_{j_i+\ell_i}^i\Big)\sup\limits_{\substack{y_i \in \re \\ i:\ell_i \leq 0}}
    \frac{|f_{j+\ell}(x+\sum_{i:\ell_i\leq 0}y_ie_i)|}{\prod\limits_{i:\ell_i\leq 0}(1+2^{\ell_i+j_i}|y_i|)^a}\Big|^p\prod\limits_{i:\ell_i>0}dx_i\prod\limits_{i:\ell_i\leq 0}dx_i\,.
    \end{split}
\end{equation}
Finally, we use the trivial estimate (already used in Lemma \ref{sharp}) to replace the operators ${\mathfrak M}_{j_i+\ell_i}^i$ by the Peetre maximal function. This gives 
\begin{equation}
  \begin{split}
   \int_{\R} |F_{j,\ell}(x)|^p\,dx &\lesssim \prod\limits_{i =1}^d \min\{2^{2p\ell_i},1\}\max\{2^{\ell_i},1\}\int_{\R} \Big|\sup\limits_{y \in \R}\frac{|f_{j+\ell}(x+y)|}{\prod\limits_{i=1}^d (1+|2^{\ell_i+j_i}y_i|)^a}\Big|^p\,dx\\
   &\lesssim \prod\limits_{i =1}^d \min\{2^{2p\ell_i},1\}\max\{2^{\ell_i},1\}\|f_{j+\ell}\|_p^p\,, 
  \end{split}
\end{equation}
if we choose $a>1/p$\,.
\end{proof}

\section{Besov-Triebel-Lizorkin spaces with mixed smoothness}
\label{perfs}

For the definition of the corresponding function spaces on $\R$ we refer to \cite{SchTr87},  \cite{Vy06}, and \cite{By18}. The corresponding function spaces on $[0,1]^d$ are defined via restrictions of functions on $\R$, see \cite[Sect.\ 3.4]{By18}. In this section mainly focus on the definition of Besov and Triebel-Lizorkin spaces with dominating mixed
smoothness on $\R$ since they are crucial for our subsequent analysis. We
closely follow \cite[Chapter 2]{SchTr87} and use the building blocks $f_j(\cdot)$ defined in \eqref{f2}.

\begin{defi}\label{d1} (Mixed Besov and Triebel-Lizorkin spaces) \\
    {\em (i)}
    Let $0< p,q\leq
    \infty$ and $r>\sigma_p:=(1/p-1)_+$. Then $S^{r}_{p,q}B(\R)$ is defined
as the collection of all $f\in
    L_1(\R)$ such that
    \begin{equation}\label{f3}
	\|f|S^r_{p,q}B(\R)\|^{\varphi}:=\Big(\sum\limits_{j\in \N_0^d}
2^{|j|_1rq}\|f_j\|_p^q\Big)^{1/q}
    \end{equation}
    is finite (usual modification in case $q=\infty$).\\
    {\em (ii)} Let $0< p<\infty$ and $r>\sigma_p:=(1/p-1)_+$. Then $S^r_{p,q}F(\R)$ is defined as
the collection of all $f\in L_1(\R)$ such that 
    $$
	\|f|S^r_{p,q}F(\R)\|^{\varphi}:=\Big\|\Big(\sum\limits_{j\in
\N_0^d}2^{|j|_1rq}|f_j(x)|^q\Big)^{1/q}\Big\|_p
    $$
    is finite. 
    \end{defi}
Recall, that this definition is independent of the chosen system $\varphi$ in
the sense of equivalent (quasi-)norms. Moreover, in case $\min\{p,q\}\geq 1$ the defined spaces are Banach spaces, whereas they are quasi-Banach spaces in case $\min\{p,q\} < 1$. For details confer \cite[2.2.4]{SchTr87}. In the next lemma there appears the condition $r>(1/p-1)_+$ which is caused by the parameter range in our Definition \ref{d1}. All the subsequent embeddings of course also hold true for the general situation, where $r\in \re$.  We have the following elementary embeddings, see \cite[2.2.3]{SchTr87}. 

\begin{lem}\label{emb} Let $0<p<\infty$, $r>(1/p-1)_+$, and $0<q\leq \infty$. Let further $A \in \{F,B\}$.
\begin{description}
   \item (i) If $\varepsilon>0$ and $0<v\leq \infty$ then
 $$
   S^{r+\varepsilon}_{p,q}A(\R) \hookrightarrow S^r_{p,v}A(\R)\,.
 $$
 \item (iia) If $p<u< \infty$, $0<q,q_1,q_2 \leq \infty$ and $r-1/p = t-1/u$ then
 \begin{equation}
   \begin{split}
     S^r_{p,q}B(\R) &\hookrightarrow S^t_{u,q}B(\R),\\
     S^r_{p,q_1}F(\R) &\hookrightarrow S^t_{u,q_2}F(\R)\,.
   \end{split}    
 \end{equation}
\item (iib) (Jawerth-Franke embedding I)  If $0<p<u<\infty$, $0<w\leq \infty$ and $r-1/p = t-1/u$ then
 $$ 
     S^r_{p,u}B(\R) \hookrightarrow S^t_{u,w}F(\R)\,.
 $$
 \item (iic) (Jawerth-Franke embedding II) If $0<p<u<\infty$, $0<q\leq \infty$ and $r-1/p = t-1/u$ then
 $$ 
     S^r_{p,q}F(\R) \hookrightarrow S^t_{u,p}B(\R)\,.
 $$
 \item (iiia) If $r>1/p$ then
 $$
    S^r_{p,q}B(\R) \hookrightarrow C(\R)\,.
 $$
 \item (iiib) if $r=1/p$ and $q\leq 1$ then we still have 
 $$
    S^{1/p}_{p,q}B(\R) \hookrightarrow C(\R)\,.
$$    
 \item (iiic) It holds
 $$
    S^1_{1,\infty}F(\R) \hookrightarrow C(\R)\,.
 $$
 \item(iv) If $1<p<\infty$ and $r>0$ then
 $$
    S^r_{p,\min\{p,q\}}B(\R) \hookrightarrow S^r_{p,q}F(\R) \hookrightarrow
S^r_{p,\max\{p,q\}}B(\R)\,.
 $$
\end{description}
\end{lem}
\begin{proof}
    The embeddings (i), (iia), (iiia), (iiib), (iv) are standard and can be found in \cite[Chapt.\ 2]{SchTr87}. As for the Jawerth-Franke type embeddings we refer to \cite[Lem.\ 3.4.2, 3.4.3]{DuTeUl18}. See also \cite[Rem.\ 3.4.4]{DuTeUl18} for further references, especially for the mixed smoothness case. Finally, the embedding in (iiic) is a consequence of (iic) and (iiib).
    \end{proof}

\subsection{Spaces on domains}
We aim for approximating functions defined on the unit cube $[0,1]^d$ with the above regularity assumptions. This requires the definition of function spaces on domains. The domain 
$\Omega\subset \R$ represents an open connected set. Later, when dealing with continuous bounded functions we may use as well compact sets like $[0,1]^d$.
\begin{defi}\label{restr}
	Let $\Omega$ be a domain in $\R$. 
    \begin{enumerate}
		\item Let additionally $0<p,q\leq \infty$ and $r>(1/p-1)_+$. Then we define 
		$$ 
		S^r_{p,q}B(\Omega):=\{f \in L_1(\Omega):\exists g\in S^r_{p,q}B(\R) \mbox{ with } g|_{\Omega}=f\}\,,
		$$
		where
		$$\|f|S^r_{p,q}B(\Omega)\| := \inf \{\|g|S^r_{p,q}B(\R)\|:g\in S^r_{p,q}B(\R),\; g|_{\Omega}=f\}.$$
		\item Let  $0<p< \infty$ and $0< q\leq \infty$ then  
		we define
		$$ 
		S^r_{p,q}F(\Omega):=\{f\in L_1(\Omega):\exists g\in S^r_{p,q}F(\R) \mbox{ with } g|_{\Omega}=f\}\,,
		$$
		where
		$$\|f|S^r_{p,q}F(\Omega)\| := \inf \{\|g|S^r_{p,q}F(\R)\|:g\in S^r_{p,q}F(\R),\;g|_{\Omega}=f\}.$$
	\end{enumerate}
\end{defi}
On bounded domains $\Omega$ all the embeddings in Lemma \ref{emb} keep valid. In addition we have the following embeddings. If $0<p_2<p_1\leq \infty$ (F-case: $p_i<\infty$), $0<q\leq \infty$ and $|\Omega|<\infty$, then
	$$ S^{r}_{p_1,q}F(\Omega)\hookrightarrow S^{r}_{p_2,q}F(\Omega)$$
	and
	$$ S^{r}_{p_1,q}B(\Omega)\hookrightarrow S^{r}_{p_2,q}B(\Omega).$$
Clearly, this is a trivial consequence of the embedding
	$$ L_{p_1}(\Omega)\hookrightarrow L_{p_2}(\Omega).$$
It is well-known that spaces $S^r_{p,q}B([0,1]^d)$ with sufficiently large smoothness, namely $r>1/p$, are compactly embedded into $L_\infty$. This is a direct consequence of results on entropy numbers of the classes $S^r_{\infty,q}B([0,1]^d)$ in $L_\infty$, see \cite[Cor.\ 23,(iii)]{MaUl21} or \cite[Thm.\ 6.2]{TeUl22} and the embeddings stated in Lemma \ref{emb} above. 

However, in case $r=1/p$ we do not have a compact embedding. For the convenience of the reader we give a direct proof. 

\begin{lem}\label{non-comp}[Non-compactness of limiting embeddings] {\em (i)} Let $0 < p < \infty$ and $q\leq 1$. Then the embedding 
\begin{equation}\label{f3S1/pB}
S^{1/p}_{p,q}B([0,1]^d) \hookrightarrow L_{\infty}([0,1]^d)
\end{equation}
is not compact.\\
{\em (ii)} If $p\leq 1$ and $0<q\leq \infty$ then the embedding
\begin{equation}\label{f4}
    S^{1/p}_{p,q}F([0,1]^d) \hookrightarrow L_\infty([0,1]^d) 
\end{equation}    
is not compact.
\end{lem}

\begin{proof} We show the non-compactness of the embedding \eqref{f3S1/pB} first in case  $d=1$. Clearly, by standard (tensorization) arguments this would imply the non-compactness also in higher dimensions. Note further, that the non-compactness of \eqref{f4} is implied by the Jawerth-Franke embedding (see Lemma \ref{emb}, (iib))  
$$
    S^{1/p_1}_{p_1,p_2}B([0,1]^d) \hookrightarrow S^{1/p_2}_{p_2,q}F([0,1]^d)
$$
together with the non-compactness of \eqref{f3S1/pB}. So, it remains to prove the non-compactness of $B^{1/p}_{p,q}([0,1]) \hookrightarrow L_\infty([0,1])$ for any $q \leq 1$ and $0<p <\infty$. 
 This fact is certainly known and can be found in the literature. However, we would like to give a direct argument here since it fits to the scope of the paper. The idea is to find a sequence $(g_j)_j$ of functions with $\|g_j|B^{1/p}_{p,q}\| \leq 1$ for all $j$  which is not convergent in $L_\infty([0,1])$\,. A straight-forward choice for the $g_j$ are the Faber-Schauder basis functions ($L_\infty$-normalized hat functions) for different levels $j$. To be more precise we let $j=0,1,2,...$ and consider the sequence $(g_j)_j := (v_{j,0})_{j\in \N}$. Clearly, we always have 
$$
    \sup\limits_{|h|\leq 2^{-k}}\|\Delta^m_h v_{j,0}\|_p \lesssim 2^{-j/p}\,.
$$
In case $k>j$ we even obtain (due to cancellation)
$$
    \sup\limits_{|h|\leq 2^{-k}}\|\Delta^m_h v_{j,0}\|_p \lesssim 2^{-k/p}2^{j-k}, 
$$
see for instance \cite[(3.19),(3.20)]{HiMaOeUl16}\,. Note, that here we have to use differences of sufficiently high order $m>s$ since $1/p$ may get large. This implies 
\begin{equation}
  \begin{split}
    \|v_{j,0}|B^{s}_{p,q}\| &\asymp \|v_{j,0}\|_1 + 
    \Big(\sum\limits_{k=0}^{\infty} [2^{ks}\sup\limits_{|h|\leq 2^{-k}}\|\Delta_h v_{j,0}\|_p]^q\Big)^{1/q}\\
    &\asymp
    2^{-j}+\Big(\sum\limits_{k\leq j} [2^{ks}2^{-j/p}]^q+
    \sum\limits_{k> j} [2^{ks}2^{-k/p}2^{j-k}]^q\Big)^{1/q}\\
    &\asymp 1\,
  \end{split}    
\end{equation}
in case $s\leq 1/p$\,. Hence, the elements $v_{j,0}$ have  uniformly (in $j$) bounded quasi-norm in $B^{s}_{p,q}$\,. However, it holds
$$
    \|g_j-g_\ell\|_{\infty}\geq 1
$$
if $j\neq \ell$, see for instance Figure \ref{fig_Faber1}\,. This directly disproves the compactness of the unit ball of $B^{1/p}_{p,q}$ in $L_\infty$.
\end{proof}

\begin{rem} We need to formulate the definition of the Kolmogorov widths to demonstrate some relations with approximative characteristics considering here. For a compact set $\mathbf{F}\hookrightarrow X$ of a Banach space $X$ define the Kolmogorov widths:
$$
d_m(\mathbf{F}, X) := \inf\limits_{\{u_i\}_{i=1}^m\subset X}
\sup\limits_{f \in \mathbf{F}}
\ \inf\limits_{c_i} 
\Big\|f-\sum\limits_{i=1}^m c_i\,u_i \Big\|_{X} , 
\ \ \ \ \ m=1,2,\dots ,
$$
and
$$
d_0(\mathbf{F}, X) := 
\sup\limits_{f \in \mathbf{F}}
\|f\|_{X} . 
$$
Clearly, considering $\mathbf{F}$ as the unit ball in  $S^{1/p}_{p,1}B([0,1]^d)$ and $X = L_{\infty}([0,1]^d)$ then 
$$
    d_m(S^{1/p}_{p,q}B([0,1]^d),L_\infty([0,1]^d) \nrightarrow 0
$$
by Lemma \ref{non-comp}.
\end{rem}

\section{The decay of the Faber-Schauder coefficients}
Now we are ready for proving our main tool: an assertion about the decay of the Faber-Schauder coefficients of functions from $S^{1/p}_{p,1}B(\R)$\,. In order to do so we need to define doubly indexed sequence spaces first.
\begin{defi}\label{defsequ} Let $0<p,q\leq \infty$ and $r\geq 1/p$.\\
{\em (i)} The doubly-index sequence space 
$s^r_{p,q}b$ is
the collection of all sequences $\{\lambda_{j,k}\}_{j\in \N_{-1}^d, k\in \Z}$
  such that
$$
  \|\lambda_{j,k}|s^r_{p,q}b\|:=\Big[\sum\limits_{j\in
\N_{-1}^d}2^{|j|_1(r-1/p)q}\Big(\sum\limits_{k\in
\Z}|\lambda_{j,k}|^p\Big)^{q/p} \Big]^{1/q}
$$
is finite.\\
{\em (ii)} Le $\Omega \subset \R$ be a compact domain. We define the index set $\D_j(\Omega)$ to be the set of all $k\in \Z$ such that $x_{j,k} \in \Omega$. The space $s^r_{p,q}b(\Omega)$ is defined as the space f all doubly-indexed sequences $\{\lambda_{j,k}\}_{j\in \N_{-1}^d,k\in \D_j(\Omega)}$ such that 
$$
\|\lambda_{j,k}|s^r_{p,q}b(\Omega)\|:=\Big[\sum\limits_{j\in
\N_{-1}^d}2^{|j|_1(r-1/p)q}\Big(\sum\limits_{k\in
\D_j(\Omega)}|\lambda_{j,k}|^p\Big)^{q/p} \Big]^{1/q}\,.
$$
\end{defi}

\begin{rem} {\em (i)} In case $\Omega = [0,1]^d$ we have $\D_j(\Omega) = \D_j$, which was defined right before Definition $\ref{def:haar_faber}$. In a certain sense the elements of $s^r_{p,q}b(\Omega)$ are restrictions of elements in $s^r_{p,q}b$ to indices related to $\Omega$.\\
{\em (ii)} These sequence spaces already appeared in \cite[Def.\ 2.1, 3.2]{Vy06}. Let $0<p,q\leq \infty$ and $r\in \re$. The
spaces $s^r_{p,q}b$ and $s^r_{p,q}b(\Omega)$ are Banach space if $\min\{p,q\}\geq 1$. In case
$\min\{p,q\} <1$ the space $s^r_{p,q}b$ is a quasi-Banach space. Moreover, if
$u:=\min\{p,q,1\}$ it is a $u$-Banach space, i.e.,
$$
    \|\lambda+\mu|s^r_{p,q}b\|^u \leq \|\lambda|s^r_{p,q}b\|^u +
\|\mu|s^r_{p,q}b\|^u\quad,\quad \lambda,\mu \in s^r_{p,q}b\,.
$$
\end{rem}

Here is the first main result.  

\begin{prop}\label{discvscont} Let $1 \leq  p < \infty$.
Then there
exists a constant $c>0$ (independent of $f$) such that
\begin{equation}\label{f7}
    \big\|d^2_{j,k}(f)|s^{1/p}_{p,\infty}b\big\| \leq c \,
\|f|S^{1/p}_{p,1}B(\R)\|
\end{equation}
for all $f\in S^{1/p}_{p,1}B(\R)$.
\end{prop}

\begin{proof}  We make use of the decomposition \eqref{f2} in a slightly modified
way. For fixed $j\in
\N_{-1}^d$ we write $f = \sum_{ \ell \in \zz^d} f_{ j+ \ell}$. Putting this
into
\eqref{f7} and using the triangle inequality yields
\begin{equation}
  \begin{split}
  \Big\|d^2_{j, k}(f)|s^{1/p}_{p,\infty}b\Big\| &\asymp 
   \sup\limits_{ j\in
   \N_{-1}^d}2^{| j|_1/p}\Big\|\sum\limits_{ k\in \Z}
d^2_{ j, k}(f)\chi_{ j, k}( x)\Big\|_p \\
   &\leq
   \sup\limits_{j\in
   \N_{-1}^d}2^{|j|_1/p}\sum\limits_{ \ell \in \zz^d}\Big\|\sum\limits_{ k\in \Z}
d^2_{j,k}(f_{j+\ell})\chi_{j,k}(x)\Big\|_p\,.\\
   \end{split}
\end{equation}
Applying Lemma \ref{lem:estfsblock} we obtain
\begin{equation}
  \begin{split}
  \Big\|d^2_{j,k}(f)|s^{1/p}_{p,\infty}b\Big\| &\lesssim 
   \sup\limits_{j\in
   \N_{-1}^d}\sum\limits_{\ell \in \zz^d}2^{|j+\ell|_1/p}\|f_{j+\ell}\|_p \prod\limits_{i =1}^d 2^{-\ell_i/p}\min\{2^{2\ell_i},1\}\max\{2^{\ell_i/p},1\}\,\\
   &\lesssim \Big(\sum\limits_{\ell \in \zz^d}2^{|j+\ell|_1/p}\|f_{j+\ell}\|_p\Big)
   \sup\limits_{\ell \in \zz^d}\prod\limits_{i =1}^d 2^{-\ell_i/p}\min\{2^{2\ell_i},1\}\max\{2^{\ell_i/p},1\} \,\\
   &\lesssim \|f|S^{1/p}_{p,1}B(\R)\|\,,
\end{split}
\end{equation}
since the $\sup$ stays finite since $1\leq p \leq \infty$\,.
\eproof

As a direct consequence we have the following restricted version
\begin{cor}\label{cor:faber} Let $1\leq p<\infty$. Then there exists a constant $c>0$ such that 
\begin{equation}\label{f7_2}
    \big\|d^2_{j,k}(f)|s^{1/p}_{p,\infty}b(\Omega)\big\| \leq c \,
\|f|S^{1/p}_{p,1}B(\Omega)\|
\end{equation}
for all $f\in S^{1/p}_{p,1}B(\Omega)$.
\end{cor}

\begin{proof} Assume $f\in S^{1/p}_{p,1}B(\Omega)$. Then, by Definition \ref{restr} there is a $g \in S^{1/p}_{p,1}B(\R)$ with $g|_{\Omega} = f$. By Proposition \ref{discvscont} we obtain 
$$
        \big\|d^2_{j,k}(g)|s^{1/p}_{p,\infty}b\big\| \leq c \,
\|g|S^{1/p}_{p,1}B\|\,.
$$
Since $d^2_{j,k}(g) = d^2_{j,k}(f)$ for $j\in \N_{-1}^d$ and $k\in\D_j(\Omega)$ we have 
$$
    \big\|d^2_{j,k}(f)|s^{1/p}_{p,\infty}b(\Omega)\big\| \leq
    \big\|d^2_{j,k}(g)|s^{1/p}_{p,\infty}b\big\|
$$
we obtain 
$$
    \big\|d^2_{j,k}(f)|s^{1/p}_{p,\infty}b(\Omega)\big\| \leq c
    \|g|S^{1/p}_{p,1}B\|\,.
$$
The last inequality holds for every extension $g$ of $f$. Taking the infimum yields the result. 
\end{proof}

\section{Application for sampling recovery in $L_q$ with $q<\infty$}
\label{application}
In this section we would like to apply the Faber-Schauder embedding in Corollary \ref{cor:faber} for sampling recovery on the unit cube $\Omega = [0,1]^d$. As we have seen in Lemma \ref{non-comp} we can not expect an error decay of a sampling recovery operator in the worst-case when we measure the error in $L_\infty([0,1]^d)$\,. Hence we focus on the recovery in $L_q([0,1]^d)$ with $q<\infty$. Based on the Faber-Schauder representation we will use a sparse grid truncation in order to obtain a recovery operator. 
Let
\be\label{Tr_F-Sch_sum}
    I_n f 
    =\sum\limits_{|j|_1\leq n} 
    \sum\limits_{k\in {\mathbb D}_j}
    d^2_{j,k}(f)v_{j,k} \,.
\ee
with the notation from Section \ref{HaarFaber}. 
\begin{lem}\label{lem:comb}
The following estimates hold true for $\alpha>0$
\begin{enumerate}
\item
\begin{equation}
    \sum_{|j|_1>n} 2^{-\alpha|j|_1}\asymp n^{d-1}2^{-\alpha n}
\end{equation}
\item and \begin{equation}
    \sum_{|j|_1\leq n} 2^{|j|_1}\asymp n^{d-1}2^{n}.
\end{equation}
\end{enumerate}
\end{lem}
\bproof
We refer to \cite[p.10, Lemma D]{Tem86}
\eproof

We first consider the case where $p=q$\,.

\begin{theorem}\label{thm:papprox}
Let $1\leq p< \infty$. Then there is a constant $C>0$ (independent of $n$ and $f$) such that
$$\|f-I_nf\|_p\leq C 2^{-n/p}n^{d-1}\|f|S^{1/p}_{p,1}B([0,1]^d)\|$$
holds for all $n\in \mathbb{N}$.
\end{theorem}
\bproof
The representation in \eqref{repr_C([0,1]^d)} allows us to express and estimate the error by

\begin{equation}
\begin{split}
\|f - I_n f\|_p
&= \Big\|\sum\limits_{|j|_1> n} 
    \sum\limits_{k\in \D_j}
    d^2_{j,k}(f)v_{j,k}\Big\|_p\\
&\leq \sum\limits_{|j|_1> n} 
    \Big\|\sum\limits_{k\in \D_j}
    d^2_{j,k}(f)v_{j,k}\Big\|_p \,2^{|j|_1/p} \, 2^{-|j|_1/p}\\ 
&\leq \sup\limits_{j} 
    2^{|j|_1/p}\Big\|\sum\limits_{k\in \D_j}
    d^2_{j,k}(f)v_{j,k}\Big\|_p \sum\limits_{|j|_1> n}  2^{-|j|_1/p}\,. 
\end{split}
\end{equation}
Finally applying Proposition \ref{discvscont} or Corollary \ref{cor:faber}, respectively, yields
\begin{equation}
\begin{split}
\|f - I_n f\|_p
&\lesssim \big\|d^2_{j,k}(f)|s^{1/p}_{p,\infty}b([0,1]^d)\big\| \, \sum\limits_{|j|_1> n}  2^{-|j|_1/p}\\
&\lesssim 2^{-n/p}n^{d-1} \|f|S^{1/p}_{p,1}B([0,1]^d)\|,\\  
\end{split}
\end{equation}
where the sum is estimated by Lemma \ref{lem:comb}.
\eproof

\begin{remark} An one-dimensional version of the above result has been proved by Oswald \cite{Os81} almost fourty years ago. 
\end{remark}

In the situation $p<q$ we loose in the main rate. This phenomenon has been observed earlier in the literature, see \cite{DuTeUl18}. We will use Jawerth-Franke type embeddings to improve the order of the logarithmic term.   
\begin{theorem}\label{thm:qapprox}Let $1\leq p<q< \infty$. Then 
$$
 \|f-I_nf\|_q\leq C 2^{-n/q}n^{(d-1)/q}\|f|S^{1/p}_{p,1}B([0,1]^d)\| .
$$
\end{theorem}
\bproof Using Jawerth-Franke I (see Lemma \ref{emb}) we obtain 
\begin{equation}
\begin{split}
\|f - I_n f\|_q & \lesssim \|f-I_nf|S^{1/p-1/q}_{p,q}B(\Omega)\|\\
&\lesssim \Big(\sum\limits_{|j|_1> n} 
    2^{q|j|_1[1/p-1/q]}\Big\|\sum\limits_{k\in \D_j}
    d^2_{j,k}(f)\chi_{j,k}\Big\|^q_p\Big)^{1/q}\\
    &\lesssim \left[\sup_{j\in \N^d_{-1}}2^{|j|_1(1/p)}\Big\|\sum\limits_{k\in \D_j}
    d^2_{j,k}(f)\chi_{j,k}\Big\|_p\right] \Big(\sum_{|j|_1>n} 2^{-|j|_1}\Big)^{\frac{1}{q}}\,,
\end{split}
\end{equation}
where we used the inverse Faber-Schauder characterization for spaces with positive smoothness, cf.  \cite[Theorem 4.18]{By18}. This reference deals with the $\R$-case but can be easily extended by standard arguments as shown for instance in the proof of Corollary \ref{cor:faber} or following the arguments in \cite[Theorem 4.25]{By18} to the unit cube setting. Finally applying Proposition \ref{discvscont} yields
\begin{equation}
\begin{split}
\|f - I_n f\|_q
    &\lesssim  2^{-n/q}n^{(d-1)/q}\|f|S^{1/p}_{p,1}B([0,1]^d)\|\,.
    \end{split}
\end{equation}
\eproof

Let us now deal with the spaces $S^1_{1,\infty}F([0,1]^d)$, which are embedded into $C([0,1]^d)$ as Lemma \ref{emb}, (iiic) shows. By Jawerth-Franke embedding II (Lemma \ref{emb}, (iic)) we even know that for every $1<p < \infty$ we have $S^1_{1,\infty}F([0,1]^d) \hookrightarrow S^{1/p}_{p,1}B([0,1]^d)$. As a direct corollary of Theorems \ref{thm:papprox}, \ref{thm:qapprox} we obtain the following

\begin{theorem}\label{thm:p=1} {\em (i)} It holds for any small $\eps>0$
$$
        \|f-I_nf\|_1 \lesssim_{\eps} 2^{-n(1-\eps)}\|f|S^1_{1,\infty}F([0,1]^d)\|\,.
$$
{\em (ii)} For any $1<q<\infty$ we have 
$$
        \|f-I_nf\|_q \lesssim  2^{-n/q}n^{(d-1)/q}\|f|S^1_{1,\infty}F([0,1]^d)\|\,.
$$
\end{theorem}
\begin{proof} By the embedding $S^1_{1,\infty}F([0,1]^d) \hookrightarrow S^{1/p}_{p,1}B([0,1]^d)$ for $p>1$ we may apply Theorem \ref{thm:qapprox} to obtain the result. For (i) we simply choose $q>1$ small and use the fact that the $L_q$-norm dominates the $L_1$-norm. Clearly, the $\log-$term can be dropped in this regime. 
\end{proof}

\subsection{Sampling numbers}
Let us sharpen our considerations to the problem of optimal sampling recovery. We compare the number $m$ of samples an algorithm is using to the resulting error and call the quantity 

\begin{equation}
    \varrho_m(S^{r}_{p,\theta}A([0,1]^d),L_q([0,1]^d)) := \inf\limits_{\substack{X_m\subset[0,1]^d,|X_m|=m\\ \varphi:\C^m\to L_q([0,1]^d)}}\sup\limits_{\|f|S^{r}_{p,\theta}A([0,1]^d)\| \leq 1}
\| f-\varphi(f(X_m))|L_q([0,1]^d)\|,
\end{equation}
$m$-th sampling width. If, in addition, the mapping $\varphi:\mathbb{C}^n \to L_q$ is linear then we obtain the linear sampling widths
\begin{equation}
    \varrho^{\rm lin}_m(S^{r}_{p,\theta}A([0,1]^d),L_q([0,1]^d)) := \inf\limits_{\substack{X_m\subset[0,1]^d,|X_m|=m\\ \varphi:\C^m\to L_q([0,1]^d)\\
    {\rm linear}}}\sup\limits_{\|f|S^{r}_{p,\theta}A([0,1]^d)\| \leq 1}
\| f-\varphi(f(X_m))|L_q([0,1]^d)\|\,.
\end{equation}

\begin{theorem}\label{thm:psamplingnumber}
Let $1\leq p<\infty$. Then  
$$\varrho_m(S^{1/p}_{p,1}B([0,1]^d),L_p([0,1]^d)) \lesssim m^{-1/p}(\log^{d-1}m)^{\frac{1}{p}+1}$$
for all $m\in\N$.
\end{theorem}
\bproof
The upper bound is due to Theorem \ref{thm:papprox} recognizing that the algorithm $I_n$ in \eqref{Tr_F-Sch_sum} samples $f$ in $m \asymp 2^nn^{d-1}$ nodes. This can be trivially checked applying Lemma \ref{lem:comb},(ii). For sake of completeness we refer to \cite[Section 5.1]{By18} where further properties of this operator were studied. 
\eproof

\begin{theorem}
Let $1\leq p <q <\infty$. Then 
\begin{enumerate}
\item $$\varrho_m(S^{1/p}_{p,1}B([0,1]^d),L_q([0,1]^d)) \lesssim m^{-1/q} (\log^{d-1}m)^{2/q} $$
and
\item $$\varrho_m(S^{1}_{1,\infty}F([0,1]^d),L_q([0,1]^d)) \lesssim m^{-1/q} (\log^{d-1}m)^{2/q} $$
\end{enumerate}
for all $m\in\N$.

\end{theorem}
\bproof
The upper bound in (i) is due to Theorem \ref{thm:qapprox} taking the number of sampling nodes into account, see Theorem \ref{thm:psamplingnumber}. The upper bound in (ii) is due to Theorem \ref{thm:p=1}.
\eproof

\subsection{Lower bounds}
The linear width of a class  $\mathbf{F}$ in a normed space $X$ has been introduced by Tikhomirov \cite{zbMATH03158232} more than $60$ years ago. It is defined by
$$
\lambda_m(\mathbf{F}, X) 
:= \inf\limits_{\substack{A:\,X\rightarrow X \ \ \text{linear} \\ \text{rank}\, A \leq m} }
\sup\limits_{f \in \mathbf{F}}
\big\|f-A(f) \big\|_{X} \, . 
$$
A.S. Romanyuk \cite{zbMATH00028594,zbMATH05268040} proved for $\mathbf{F}$, the unit ball in $S^{1/p}_{p,1}B([0,1]^d)$, that in case $1\leq p\leq q\leq 2$ and $q>1$
\begin{equation}
  \begin{split}
        \lambda_m(S^{1/p}_{p,1}B([0,1]^d), L_q([0,1]^d)) &\geq d_m(S^{1/p}_{p,1}B([0,1]^d), L_q([0,1]^d))\\
        &\gtrsim (m^{-1} \log^{d-1}m)^{1/q}\,.
\end{split}
\end{equation}
From these bounds we obtain the following lower bounds for the (linear) sampling numbers.
\begin{theorem} Let $1\leq p\leq q<\infty$. Then we have\\
{\em (i)}
$$
    \varrho_n(S^{1/p}_{p,1}B([0,1]^d),L_q([0,1]^d)) \gtrsim n^{-1/q}\,.
$$
{\em (ii)} If additionally $1\leq p\leq q\leq 2$ then 
$$  
 \varrho^{\rm lin}_n(S^{1/p}_{p,1}B([0,1]^d),L_q([0,1]^d)) \gtrsim (n^{-1}\log^{d-1} n)^{1/q}\,.
$$
\end{theorem}
\section{Outlook and discussion}
In this paper we have shown that for the sampling recovery problem the compact embedding into $L_\infty$ is not necessary. There are several relevant multivariate function spaces which fall under this scope, like limiting Besov and Triebel-Lizorkin spaces with smoothness $r=1/p$ and further parameter conditions to ensure the embedding into the space of continuous functions. We were able to give upper bounds for the so-called sampling numbers in $L_q$ with $q<\infty$ which are sharp in the polynomial main rate. What concerns the right order of the logarithm the situation is completely open. The results presented here are, as far as the authors know, the first results in this direction apart from a univariate result of Oswald \cite{Os81}.

Let us comment on the particular case of $L_1$-smoothness spaces. Smoothness spaces built upon $L_1$ with smoothness $r=1$ play an important role for numerical integration. This includes for instance the space $S^1_1W([0,1]^d)$ defined via weak derivatives. This space can not be described via Fourier analytical means and is therefore difficult to handle. See for instance $\cite{LueUl22}$\,. However, these spaces fall into the scope of this paper since we also have $r=1/p$ and a non-compact embedding into $L_\infty$. In \cite{LueUl22} a Faber-Schauder characterization is shown, similar as above, namely 
\begin{equation}\label{S11W}
    \sup\limits_{j\in \N_{-1}^d} \sum\limits_{k\in \zz^d} |d^2_{j,k}(f)| \lesssim \|f|S^1_1W(\R)\|\,.
\end{equation}
In particular, this would imply the following so-called sampling inequality
$$
    \sum\limits_{k\in \zz^d} |f(k)| \lesssim \|f|S^1_1W(\R)\|\,.
$$
This extends the result in \cite[Prop.\ 2]{SchSi2000} in several directions. On the one hand we consider the multivariate case and on the other hand the space $S^1_1W(\R)$ is larger than $S^1_{1,1}B(\R)$\,.

Having \eqref{S11W} at hand the following slightly sharper version of Theorem \ref{thm:p=1} is  immediate.
\begin{theorem}\label{thm:S11W} {\em (i)} It holds
$$
        \|f-I_nf\|_1 \lesssim 2^{-n}n^{d-1}\|f|S^1_1W([0,1]^d)\|\,.
$$
{\em (ii)} For any $1<q<\infty$ we have 
$$
        \|f-I_nf\|_q \lesssim  2^{-n/q}n^{(d-1)/q}\|f|S^1_1W([0,1]^d)\|\,.
$$
\end{theorem}
\noindent A proper Fourier analytical replacement of the spaces $S^1_1W([0,1]^d)$ are the spaces $S^1_{1,q}F([0,1]^d)$. However these spaces are not really comparable, especially when $q=\infty$. Using the methods in \cite[Thm.\ 1.9]{GaSeUl22} there is strong evidence for proving a version of \eqref{S11W} also for the spaces $S^1_{1,\infty}F(\R)$. This would imply a version of Theorem \ref{S11W}. By well-known arguments a cubature formula with performance
$$
    \Big|f-\sum\limits_{i=1}^N \lambda_i f(x^i)\Big| \lesssim N^{-1}\log N^{2(d-1)}\|f|S^1_{1,\infty}F([0,1]^d)\|
$$
can be constructed by integrating the approximand. Note, that there were efforts in the literature to treat such limiting cases, see for instance \cite[Cor.\ 6.5]{UlUl16}. Suboptimal bounds were proven there. Note, that results for such limiting cases are related to the Kokhsma-Hlawka inequality, saying that QMC-cubature in $S^1_1W$ is related to the star discrepancy of the cubature nodes.

\nocite{TDiss07,NaSchUl21,LiTe22,Te21,TeUl22,TeUl21,PoUl22}
\vspace{1cm}
{\bf Acknowledgments}
  Special thanks go to Winfried Sickel for stimulating discussions. 
  The second named author is supported
  by the German Academic Exchange Service (DAAD, Grant 57588362)
  and
  by the Philipp Schwartz initiative of the Alexander von Humboldt Foundation.

\end{document}